\documentclass{amsart}
\usepackage{graphicx}
\usepackage{amssymb}
\newtheorem{thm}{Theorem}[section]
\newtheorem{cor}[thm]{Corollary}
\newtheorem{lemma}[thm]{Lemma}
\newtheorem{ex}[thm]{Example}
\newtheorem{theorem}[thm]{Theorem}
\newtheorem{prop}[thm]{Proposition}
\theoremstyle{definition}

\newtheorem{remark}[thm]{Remark}
\theoremstyle{question}

\theoremstyle{Conjecture}

\numberwithin{equation}{section}
\begin{document}

\title[A generalization of Neumann's question]{A generalization of Neumann's question}%
\author{N. Ahmadkhah, S. Marzang and M. Zarrin}%

\address{Department of Mathematics, University of Kurdistan, P.O. Box: 416, Sanandaj, Iran}%
 \email{N.ahmadkhah@sci.uok.ac.ir \\S.marzang@sci.uok.ac.ir \text{and} M.zarrin@uok.ac.ir}
% ----------------------------------------------------------------
\begin{abstract}
Let $G$ be a group, $m\geq2$ and $n\geq1$. We say that $G$ is an $\mathcal{T}(m,n)$-group if for every $m$ subsets $X_1, X_2, \dots, X_m$ of $G$ of cardinality $n$, there exists $i\neq j$ and $x_i \in X_i, x_j \in X_j$ such that $x_ix_j=x_jx_i$. In this paper, we give some examples of finite and infinite non-abelian $\mathcal{T}(m,n)$-groups and we discuss finiteness and commutativity of such groups. We also show solvability length of a solvable $\mathcal{T}(m,n)$-group is bounded in terms of $m$ and
$n$.
\end{abstract}
\maketitle
% ----------------------------------------------------------------

\section{\textbf{ Introduction}}

Let $m, n$ be positive integers or infinity (denoted $\infty$) and $\mathcal{X}$ be a class of groups.
We say that a group $G$ satisfies the condition $\mathcal{X}(m, n)$ ($G$ is an $\mathcal{X}(m, n)$-group,
or $G\in \mathcal{X}(m, n)$), if for every two subsets $M$
and $N$ of cardinalities $m$ and $n$, respectively, there exist $x\in M$ and $y\in N$ such that
$\langle x, y\rangle\in \mathcal{X}$.  Bernhard H. Neumann in 2000 \cite{neu2}, put forward the question: Let $G$ be a finite group of order $|G|$ and assume that however a set $M$ of $m$ elements and a set $N$ of $n$ elements of the group is chosen, at least one element of $M$ commutes with at least one element of $N$, that is $G$ is an $\mathcal{C}(m, n)$-group, where $\mathcal{C}$ is the class of abelian groups. What relations between $|G|$, $m$ and $n$ guarantee that $G$ is abelian?  Even though the
latter question was posed for finite groups, the property introduced therein can be
considered for all groups.

Following Neumann's question, authors in \cite{abd1}, showed that infinite groups satisfying the condition $\mathcal{C}(m,n)$ for some $m$ and $n$ are abelian. They obtained an upper bound in terms of $m$ and $n$ for the solvability length of a solvable group $G$. Also the third author in \cite{zar2} studied the $\mathcal{N}(m,n)$-groups, where $\mathcal{N}$ is the class of nilpotent groups. Considering the analogous question for rings, Bell and Zarrin in \cite{bel1} studied the $\mathcal{C}(m,n)$-rings and they showed that all infinite
$\mathcal{C}(m, n)$-rings (like infinite
$\mathcal{C}(m, n)$-groups) are commutative and proved several commutativity results.

As a substantial generalization of $\mathcal{C}(m,n)$-rings, Bell and Zarrin in \cite{bel2} studied the $\mathcal{T}(m,n)$-rings. Let $m\geq2$ and $n\geq1$. A ring $R$ (or a semigroup) is said to be a $\mathcal{T}(m,n)$-ring (or is an $\mathcal{T}(m,n)$-semigroup), if for every $m~~ n$-subsets $A_1, A_2, \dots, A_m$ of $R$, there exists $i\neq j$ and $x_i \in A_i, x_j \in A_j$ such that $x_ix_j=x_jx_i$.  They showed that torsion-free $\mathcal{T}(m,n)$-rings are commutative. Note that unlike $\mathcal{C}(m,n)$-rings there are the vast classes of infinite noncommutative $\mathcal{T}(m,n)$-rings. Also they
discussed finiteness and commutativity of such rings.

 In this paper, considering the analogous definition for groups, we prove some results for $\mathcal{T}(m,n)$-groups and present some examples of such groups and give several commutativity theorems. Note that infinite $\mathcal{T}(m,n)$-groups,
unlike infinite $\mathcal{C}(m,n)$-groups, need not be commutative.  For instance, we can see $A_5\times A$ is an infinite non-Abelian $\mathcal{T}(22,1)$-group, where $A_5$ is the alternating group of degree 5 and $A$ is an arbitrary infinite abelian group. However, certain infinite $\mathcal{T}(m,n)$-groups can be shown to be
commutative. Clearly, every finite group is an $\mathcal{T}(m,n)$-group, for some $m$ and $n$.  It is not necessary that every group is an $\mathcal{T}(m,n)$-group, for some $m\geq 2, n\geq 1$. For example, if $F$ be a \textbf{free} group, then it is not difficult to see that $F$ is not an $\mathcal{T}(m,n)$-group, for every $m\geq 2, n\geq 1$. \\
It is easy to see that every $\mathcal{C}(m,n)$-group is an $\mathcal{C}(\max \{m, n\},\max \{m, n\})$-group and every $\mathcal{C}(\max \{m, n\},\max \{m, n\})$-group is an $\mathcal{T}(2,\max \{m, n\})$-group. Therefore every $\mathcal{C}(m,n)$-group is an $\mathcal{T}(2, r)$-group for some $r$. Thus  a next step might be to consider $\mathcal{T}(3, r)$-groups. We show solvability length of a solvable $\mathcal{T}(3,n)$-group is bounded above in terms of $n$. Also we give a solvability criterion for $\mathcal{T}(m,n)$-groups in terms of $m$ and $n$.\\
 Finally, in view of $\mathcal{T}(m,n)$-groups and $\mathcal{X}(m, n)$-groups, we can give a substantial generalization of $\mathcal{X}(m, n)$-groups.
  Let $m, n$ be positive integers or infinity (denoted $\infty$) and $\mathcal{X}$ be a class of groups.
We say that a group $G$ satisfies the condition $\mathcal{GX}(m, n)$ (or $G\in \mathcal{GX}(m, n)$), if for every $m, n$-subsets $A_1, A_2, \dots, A_m$ of $G$, there exists $i\neq j$ and $x_i \in A_i, x_j \in A_j$ such that $\langle x_i, x_j\rangle\in \mathcal{X}$. Also a set $\{A_1, A_2, \dots, A_m\}$ of $n$-subsets of a group $G$ is called $(m,n)$-obstruction if it prevents $G$ from being an $\mathcal{GX}(m,n)$-group.
Therefore, with this definition, $\mathcal{T}(m,n)$-groups are exactly $\mathcal{GC}(m,n)$-groups. Obviously, every $\mathcal{X}(m,n)$-group is an $\mathcal{GX}(m+n, 1)$-group.

\section{\textbf{ Some properties of $\mathcal{T}(m,n)$-groups}}

Here, we use
the usual notation, for example $A_n, S_n, SL_n(q), PSL_n(q)$ and $Sz(q)$, respectively,
denote the alternating group on $n$ letters, the symmetric group on $n$ letters, the special linear group of degree $n$ over the finite field of size $q$,
the projective special linear group of degree $n$ over the finite field of size $q$ and
the Suzuki group over the field of size $q$.

At first, we give some properties of $\mathcal{T}(m,n)$-groups and then give some examples of such groups.
If $2\leq m_1\leq m_2$, $n_1\leq n_2$, then every $\mathcal{T}(m_1,n_1)$-group is a $\mathcal{T}(m_2,n_2)$-group. Therefore  every $\mathcal{T}(m,n)$-group is an $\mathcal{T}(\max\{m, n\},\max\{m, n\})$-group.
\begin{lemma}
If $G$ is an $\mathcal{T}(m,n)$-group, $H\leq G$ and $N\unlhd G$, then two groups $H$ and $\frac{G}{N}$ are $\mathcal{T}(m,n)$-group.
\end{lemma}

\begin{lemma}
Let $m+n\leq 4$, then $G$ is an $\mathcal{T}(m,n)$-group if and only if it is abelian.
\end{lemma}

\begin{proof}
It is enough to consider only the groups that belongs to $\mathcal{T}(3,1)$ and $\mathcal{T}(2,2)$. If $G$ be a non-Abelian $\mathcal{T}(3,1)$-group, then there exist elements $x$ and $y$ of $G$, such that $[x,y]\neq 1$. Therefore $A_1=\{x\}, A_2=\{y\}, A_3=\{xy\}$ is a $(3,1)$-obstruction of $G$, a contrary.

If $G$ is a non-Abelian $\mathcal{T}(2,2)$-group and $[x,y]\neq1$. Then we can see that $A_1=\{x,y\}, A_2=\{xy,yx\}$ is a $(2,2)$-obstruction of $G$, a contrary.
\end{proof}
 We note that the bound 4 in the above Lemma is the best possible. As  $D_8$ and $Q_8$ are $\mathcal{T}(4,1)$-groups.

\begin{prop}\label{p1}
Assume that a finite group $G$ is not $\mathcal{T}(m,n)$-group, for two positive integers $m,n$. Then  $mn\leq |G|-|Z(G)|$.
Moreover, if $mn=|G|-|Z(G)|$, then for every $a\in G\backslash Z(G)$, $|a|\leq n+|Z(G)|$.
\end{prop}
\begin{proof}
 As $G$ is not an $\mathcal{T}(m,n)$-group, then there exists $(m,n)$-obstruction for $G$ like $\{A_1,A_2,\ldots,A_m\}$. It follows that $A_i \bigcap A_j={\O}$ for every $i\neq j$ and $Z(G)\bigcap A_i={\O}$. Therefore $\bigcup_{i=1}^{m}A_i \subseteq G\backslash Z(G)$ and so $mn\leq |G|-|Z(G)|$.\\
Now if $mn=|G|-|Z(G)|$, then we can see that, for every noncentral element $a\in G$, there exists $1\leq i\leq m$ such that $C_G(a)\backslash Z(G)\subseteq A_i$. Thus $|C_G(a)|\leq n+|Z(G)|$ and  so $|a|\leq n+|Z(G)|$.
\end{proof}

\begin{cor}
$S_3$ is an $\mathcal{T}(2,3)$ and $\mathcal{T}(3,2)$-group. Also the dihedral group of order $2n$, $D_{2n}$ is an  $\mathcal{T}(2,n)$. If n is even integer, then $D_{2n}$ is an $\mathcal{T}(2,n-1)$ but not an $T(2,n-2)$-group. For this, if we put $D_{2n}=\langle a,b\mid a^n=b^2=1, a^b=a^{-1}\rangle$, $A_1=\langle a \rangle \backslash Z(D_{2n})$ and $A_2=\bigcup_{j=1}^{{\frac {n}{2}}-1}C_G(ba^j)\backslash Z(D_{2n})$. Then we can see that $\{A_1, A_2\}$ is a $(2,n-2)$-obstruction. Moreover it is easy to see that, if $n$ is even integer, then $D_{2n}$ is an  $\mathcal{T}(\frac {n}{2}+2,1)$-group. Also, if $n$  is odd integer, then $D_{2n}$ is an $\mathcal{T}(n+2,1)$-group. Finally, every finite $p$-group of order $p^n$ is an $\mathcal{T}(p^{n-1},p)$ and also an $\mathcal{T}(p,p^{n-1})$-group.
\end{cor}
\begin{ex}
Let $G=S_4$. It is not difficult to see that,  according to the centralizers of $G$, $G$ is an $\mathcal{T}(14,1)$, $\mathcal{T}(11,2)$, $\mathcal{T}(6,3)$, $\mathcal{T}(4,5)$ and $\mathcal{T}(3,7)$-group.
\end{ex}

\begin{lemma}
 Every finite group $G$ is an $\mathcal{T}(m,\lceil\frac{|G|}{m}\rceil)$ and so is an $\mathcal{T}(m,\lceil\frac{|G|}{2}\rceil)$, for every $m\geq 2$.
\end{lemma}
\begin{proof}
 The result follows from Proposition \ref{p1}.
\end{proof}

\begin{remark}
 Assume that $G_1$ is an $\mathcal{T}(m_1,n_1)$-group and $G_2$ is an $\mathcal{T}(m_2,n_2)$-group. Then the group $G_1\times G_2$ need not to be an $\mathcal{T}(m,n)$-group, where $m=\max\{m_1,m_2\}$ and $n=\max\{n_1,n_2\}$.
For example, $S_3$ is an $\mathcal{T}(3,2)$-group but $S_3\times S_3$ is not an $\mathcal{T}(3,2)$-group (note that $S_3\times S_3$ is an $\mathcal{T}(7,3)$-group).
In particular,  the group $G_1\times G_2$ need not to be even an $\mathcal{T}(m_1m_2,n_1n_2)$-group.
For example, it is easy to see that the quaternion group $Q_8$ is an $\mathcal{T}(4,1)$-group, but $Q_8\times S_3$ is not an $\mathcal{T}(12,2)$-group (in fact, $Q_8\times S_3$ is an $\mathcal{T}(13,2)$-group). For, if we consider the subsets of $Q_8\times S_3$ as follows:

$$A_1=\{(i,(1,2)),(-i,(1,2))\},~ A_2=\{(i,(1,3)),(-i,(1,3))\},$$

$$A_3=\{(i,(2,3)),(-i,(2,3))\},A_4=\{(i,(1,2,3)),(-i,(1,2,3))\},$$

$$A_5=\{(j,(1,2)),(-j,(1,2))\}, A_6=\{(j,(1,3)),(-j,(1,3))\},$$

$$ A_7=\{(j,(2,3)),(-j,(2,3))\}, A_8=\{(j,(1,2,3)),(-j,(1,2,3))\},$$

$$A_9=\{(k,(1,2)),(-k,(1,2))\}, A_{10}=\{(k,(1,3)),(-k,(1,3))\},$$

$$ A_{11}=\{(k,(2,3)),(-k,(2,3))\}, A_{12}=\{(k,(1,2,3)),(-k,(1,2,3))\}.$$

 Then it is easy to see that the subsets $\{A_1, A_2, \dots, A_{12}\}$ is a $(12,2)$-obstruction for the group $Q_8\times S_3$.

\end{remark}

\begin{lemma}
Let $G_1$ be an $\mathcal{T}(m,1)$-group, $G_2$ be an abelian group, then $G_1\times G_2$ is an $\mathcal{T}(m,1)$-group.
\end{lemma}

\begin{remark}
Clearly every finite group is an $\mathcal{T}(m,n)$-group for some $m, n$. But  it is not true that every infinite group is an $\mathcal{T}(m,n)$-group. For instance, every group which contain a free subgroup, is not an $\mathcal{T}(m,n)$-group, for every $m\geq 2, n\geq 1$. Moreover, it is well-known that every free group is a residually finite group (even though the converse in not necessarily true). But there exist some residually finite groups that are not again an $\mathcal{T}(m,n)$-group, for every $m\geq 2, n\geq 1$. For example, the group $SL_2(Z)$ is a residually finite group. The subgroup of $SL_2(Z)$ generated by matrixes $a=\begin{pmatrix}1 & 2\\0 & 1\\\end{pmatrix}$ and $b=\begin{pmatrix}1 & 0\\2 & 1\\\end{pmatrix}$ is a free group of rank 2. So $SL_2(Z)$ is not an $\mathcal{T}(m,n)$-group, for every $m\geq 2$ and $n\geq 1$.
\end{remark}

For any nonempty set $X$, $|X|$ denotes the cardinality of $X$.  Let $A$ be a subset of a group $G$. Then a subset $X$ of $A$ is a set of pairwise non-commuting elements if $xy\neq yx$ for any two distinct elements $x$ and $y$ in $X$. If $|X|\geq |Y|$ for any other set of pairwise non-commuting elements $Y$ in $A$, then the cardinality of $X$ (if it exists) is denoted by $w(A)$ and is called the clique number of $A$ (for more information concerning the clique number of groups, see for example \cite{zar5} and \cite{abd2}).

\begin{lemma}
Let $G$ is not an $\mathcal{T}(m,n)$-group and $\{A_1, A_2, \dots, A_m\}$ be a $(m,n)$-obstruction for $G$. Then $$m+\max\{w(A_i)\mid 1\leq i\leq m\}\leq w(G).$$
\end{lemma}
\begin{proof}
Clearly.
\end{proof}
\begin{lemma}
Let $G$ be an $\mathcal{T}(m,n)$-group. Then $w(G) < mn$ and $G$ is center-by-finite.
\end{lemma}
\begin{proof}
We show that for any set $X$ of pairwise non-commuting elements of $G$, we have $|X|<mn$. Suppose that $|X|\geq mn$, then we can take $m$ n-subsets of $X$ that is a $(m,n)$-obstruction for $G$. It is a contradiction.
By the famous theorem of B. H. Neumann \cite{neu1}, since every set of non-commuting elements of $\mathcal{T}(m,n)$-group $G$ is finite, therefore it is center-by-finite.
\end{proof}
Now we show that for $\mathcal{T}(m,n)$-groups with $|Z(G)|\geq n$, we get even $w(G)<m$. In fact, we have

\begin{prop}\label{p2}
If $G$ is an $\mathcal{T}(m,n)$-group, then $|Z(G)|<n$ or $w(G)<m$.
\end{prop}
\begin{proof}
Let $G$ be an $\mathcal{T}(m,n)$-group and $|Z(G)|\geq n$. We may assume $Z_1\subseteq Z(G)$ and $|Z_1|=n$. Now if $w(G)\geq m$ and $\{x_1, x_2, \dots, x_m\}$ be a pairwise non-commuting set of $G$, then $\{x_1Z_1,x_2Z_1,\ldots,x_mZ_1\}$ is a $(m,n)$-obstruction for $G$, which is a contradiction.
\end{proof}

It is easy to see that a group $G$ is an $\mathcal{T}(m,1)$-group, if and only if $w(G)<m$.
\begin{cor}
Assume that $G$ is a nilpotent finite $\mathcal{T}(m,n)$-group and $p$ is a prime divisor of $|G|$ such that $n\leq p$. Then $G$ is an $\mathcal{T}(m,1)$-group.  In particular, every nilpotent $\mathcal{T}(m,2)$-group is an $\mathcal{T}(m,1)$-group.

If $G$ is a non-Abelian group, then $G$ is not an $\mathcal{T}(3,z)$, which $z=|Z(G)|$, since $w(G)\geq 3$.

If $G$ is an $\mathcal{T}(w(G),2)$-group, then $Z(G)=1$.
\end{cor}

\begin{cor}
Let $G$ be a non-Abelian $\mathcal{T}(m,n)$-group with at least $m$ pairwise non-commuting elements, then $G$ is a finite group.
\end{cor}

\begin{lemma}
Let $G$ be a non-Abelian $\mathcal{T}(2,n)$ or $\mathcal{T}(3,n)$-group and $N$ be a normal subgroup of $G$ such that $G/N$ is non-Abelian. Then $|N|<n$.
\end{lemma}
\begin{proof}
Suppose that $N=\{a_1,a_2,\ldots,a_t\}$ and $t\geq n$. It is enough to prove the theorem for non-Abelian $\mathcal{T}(3,n)$-groups. We chose elements $x,y$ in $G\setminus N$, and we consider three subsets of $G$, as follows: $$A_1=\{xa_1,xa_2,\dots,xa_t\},~~A_2=\{ya_1,ya_2,\dots,ya_t\}$$ and $$A_3=\{xya_1,xya_2,\dots,xya_t\}.$$ Now as $G$ is an $\mathcal{T}(3,n)$-group, we can follow that $[x,y]\in N$, that is  $G/N$ is abelian, which is a contradiction.
\end{proof}

\begin{thm}\label{t1}
Let $G$ be a  non-Abelian group and its clique number is finite. Then there exist a natural number $m$ such that $G$ is an $\mathcal{T}(m,n)$-group for all $n\in N$.
\end{thm}
\begin{proof}
 As the clique number of $G$ is finite, so according to the famous Theorem of B. H. Neumann \cite{neu1}, $G$ is center-by-finite. So we put $[G:Z(G)]=m$. We claim that for every $n\in N$, $G$ is an $\mathcal{T}(m,n)$-group. There exists $m-1$ elements $g_1, g_2,\dots,g_{m-1}$ in $G$, such that $Z(G), g_1Z(G),g_2Z(G),\ldots,g_{m-1}Z(G)$ are distinct cosets of $Z(G)$ in $G$ and $$G=Z(G)\bigcup(\bigcup_{j=1}^{m-1}(g_jZ(G))).$$ Let $\{A_1, A_2,\ldots,A_m\}$ be an $(m,n)$-obstruction of $G$. Now as every $g_iZ(G)$ is abelian, therefore if $g_iZ(G)\bigcap A_r\neq \emptyset$ for some $1\leq i \leq m-1$ and $1\leq r \leq m$, then $g_iZ(G)\bigcap A_j= \emptyset$ for every $j\neq r$. On the other hand $Z(G)\bigcap A_i=\emptyset$ for every $1\leq i \leq m$, thus $A_i\subseteq \bigcup_{j=1}^{m-1} (g_jZ(G))$ for every $1\leq i \leq m$. From this one can follow that there exist two subsets like $A_r$ and $A_s$ such that $A_r\cup A_s \subseteq g_jZ(G)$ for some $1\leq j \leq m-1$, a contradiction. Therefore  $G$ is an $\mathcal{T}(m,n)$-group, for all $n\in N$.
\end{proof}

\begin{remark}
In the above Theorem, the finiteness of clique number is necessary. For example, it should be borne in mind that infinite $p$-groups can easily have trivial
center. The group $G=C_p\wr C_{p^{\infty}}$, the regular wreath product $C_p$ by $C_{p^{\infty}}$,  is an infinite centerless $p$-group, where $C_p$ is a cyclic group of order $p$ and $C_{p^{\infty}}$ is a quasi-cyclic (or Pr$\ddot{u}$fer) group. So $[G:Z(G)]$ is infinite and so  $w(G)$ is infinite and $G$ is not an $\mathcal{T}(m,n)$-group.
\end{remark}
B. H. Neumann \cite{neu1} showed that if every set of non-commuting elements of group $G$ is finite, then $G$ is center-by-finite. Moreover, it is not difficult to see that every center-by-finite group has finite clique number. Here, by using the above Theorem, we will obtain the following result.
\begin{cor}
If $G$ be a group and $[G:Z(G)]=m$. Then $w(G)< m$.
\end{cor}
\begin{proof}
As $[G:Z(G)]=m$  and $G$ is an $\mathcal{T}(m,1)$-group, if and only if $w(G)<m$, the result follows by Theorem \ref{t1}.
\end{proof}

\begin{cor}
Every infinite  $\mathcal{T}(m,n)$-group with $m\leq 3$, is an abelian group.
\end{cor}

\begin{proof}
 If $G$ is non-Abelian group, then there exists $x,y$ such that $xy\neq yx$. So $\{x,y,xy\}$ is a subset of pairwise non-commuting elements of $G$. Therefore by Proposition \ref{p2}, $|Z(G)|<n$ and by Lemma $2.11$ $G$ is center-by-finite and so $G$ is a finite group, a contradiction.
\end{proof}
\begin{thm}
Let $G$ be a finite  $\mathcal{T}(m,n)$-group, $m\leq 4$, $n>1$ and $(p,|G|)=1$, for every prime number $p\leq n$. Then $G$ is abelian.
\end{thm}

\begin{proof}
It is enough to prove the theorem for the case $m=4$. Suppose, a contrary, that $G$ is a non-Abelian $\mathcal{T}(4,n)$-group. Then there exists elements $x$ and $y$ in $G$, such that $xy\neq yx$. Now we consider four subsets of $G$ as follows:  $$A_1=\{x,x^2,\ldots,x^n\},~~A_2=\{y,y^2,\ldots,y^n\},$$ $$A_3=\{xy,(xy)^2,\ldots,(xy)^n\} \text{~~and~} A_4=\{xy,xy^2,\ldots,xy^{n-1},x^2y\}.$$ Then it is not difficult to see that $\{A_1, A_2, A_3, A_4\}$ is a $(4,n)$-obstruction for $G$, a contradiction (note that if $a\in G$ and $(i, |a|)=1$, then  $C_G(a^i) = C_G(a)$).
\end{proof}
As a corollary, for $p>2$ every finite $p$-group, $G\in \mathcal{T}(4,p-1)$ is abelian.

Note that the group $D_8$ is a non-Abelian $\mathcal{T}(4,1)$-group. This example suggests that it may be necessary to restrict
ourselves to  $\mathcal{T}(m,n)$-groups with $n>1$ in the above Theorem.

\begin{prop}
Assume that $G$ is a non-Abelian group. Then
\begin{itemize}
\item [(1)]If there exist a positive integer $n$ such that $(p,|G|)=1$, for every prime number $p\leq n$. Then $w(G)\geq n+2$.
\item [(2)]If $p$ is the smallest prime divisor of $|G|$. Then $w(G)\geq p+1$. Moreover,
if $G$ is a finite $p$-group and $G\in\mathcal{T}(m,p-1)$, then $p+1\leq w(G)<m$.
\end{itemize}
\end{prop}

\begin{proof}
(1)~~Since $G$ is non-Abelian group, there exists elements $x,y$ in $G$, such that $xy\neq yx$. Now $X=\{x,y,xy,xy^2,\ldots,xy^{n}\}$ is a set of pairwise non-commuting elements of $G$ of cardinality $n+2$.\\
(2)~~For every prime number $q\leq p-1$, $(q,|G|)=1$, then by part $(1)$, $w(G)\geq (p-1)+2$. Thus $w(G)\geq p+1$.
\end{proof}

\begin{lemma}
Let $G$ be a finite $\mathcal{T}(m,n)$-group where $Z(G)\neq 1$ and $p$ be smallest prime divisor of $|G|$. Then\\
(1)~~ $p\leq max\{m-2,n-1\}$.\\
(2)~~If $G$ is a nilpotent, then $p\leq max\{m-2,\sqrt[t]{n-1}\}$, where $t$ is the number of prime divisors of order $G$, $|\pi(G)|$.
\end{lemma}

\begin{proof}
(1)~~As $G$ is an $\mathcal{T}(m,n)$-group, by Proposition \ref{p2}, $p\leq |Z(G)|<n$ or $p+1\leq w(G)<m$. Therefore $p\leq n-1$ or $p\leq m-2$ and so $p\leq max\{m-2,n-1\}$.

(2)~~In this case it is enough to note that the set of prime divisors of the center of $G$ is equal to $\pi(G)$, so $p^t\leq \prod_{i=1}^{t}p_i\leq |Z(G)|<n$. Thus $p\leq \sqrt[t]{n-1}$ or $p\leq m-2$.
\end{proof}

\begin{cor}\label{co1}
(1)~~If $G$ is a finite $p$-group and $G\in \mathcal{T}(p,p)$, then $G$ is abelian.\\
(2)~~Every finite non-Abelian nilpotent $\mathcal{T}(m,n)$-group with $3\leq n\leq 6$ and $m\leq 3$ is a $p$-group.\\
(3)~~If $G$ is a finite nilpotent $\mathcal{T}(4,n)$-group with $n\leq 3$ and odd order, then $G$ is an abelian group.\\
(4)~~ Every finite nilpotent $\mathcal{T}(4,n)$-group with $n\leq 3$, is an abelian-by-$2$-group.
\end{cor}

\begin{theorem}\label{t2}
Let $G$ be a non-Abelian nilpotent $\mathcal{T}(3,n)$-group. Then $$|\pi(G)|\leq log_3(n+2).$$
\end{theorem}
\begin{proof}
Use induction on $|\pi(G)|$, the case $3\leq n\leq 6$ being clear by Case (2) of Corollary \ref{co1}. Assume that $n\geq 7$ and the result holds for $|\pi(G)|-1$. Since  $G$ is finite non-Abelian nilpotent, then there exist a Sylow subgroup $P$ of $G$, such that $\frac{G}{P}$ is non-Abelian and  $\frac{G}{P}\in \mathcal{T}(3,n-t)$-group, for every $2t\leq n$. So $|\pi(\frac{G}{P})|\leq log_3(n-t+2)$, therefore $|\pi(G)|-1\leq log_3(n-t+2)< log_3(n+2)$ and hence $|\pi(G)|\leq log_3(n+2)$, as wanted.
\end{proof}

\begin{remark}
By argument similar to the one in the proof of Theorem \ref{t2}, we can follow that if $G$ is a non-Abelian nilpotent $\mathcal{T}(4,n)$-group with odd order, then $|\pi(G)|\leq log_4(n+6)$ (in this case note that if $4\leq n\leq 9$, then $G$ is a $p$-group, for some prime number $p$).
\end{remark}

\section{\textbf{On solvable $\mathcal{T}(m,n)$-groups}}

In this section we investigate solvable $\mathcal{T}(m,n)$-groups.  At first we  obtain the derived length of a solvable $\mathcal{T}(m,n)$-group in terms $n$, for $m=3$ or $4$ and then give a solvability criterion for $\mathcal{T}(m,n)$-groups in terms $m$ and $n$.  To prove our results  it is necessary to establish a technical lemma.

\begin{lemma}\label{l32}
Let $G$ be an $\mathcal{T}(m,n)$-group for some integers $m\geq 2, n>1$ and $N$ be a proper non-trivial normal subgroup of $G$, then  $\frac{G}{N}$ is an $\mathcal{T}(m,n-t)$-group, where $n\geq2t$.
\end{lemma}
\begin{proof}
Suppose that $G$ is an $\mathcal{T}(m,n)$-group and $N\lhd G$, but $\frac{G}{N}$ is not an $\mathcal{T}(m,n-t)$-group. We can take $m$ subsets $X_i=\{x_{i1}N,x_{i2}N, \ldots,x_{in-t}N\}$, $1\leq i\leq m$ of $\frac{G}{N}$ of cardinality $n-t$, such that for every $1\leq i,j\leq m$ and $1\leq k,l\leq n-t$, $[x_{ik}, x_{jl}]$ is not belongs to $N$. Let $a$ be a nontrivial element of $N$, then we can obtain $m$ $n$-subsets $Y_i=\{ax_{i1}, ax_{i2}, \dots, ax_{in-t}, x_{i1}, x_{i2},\ldots,x_{it}\}$ of $G$, for some $2t\leq n$. Thus $\{Y_1, Y_2, \ldots, Y_m\}$ is a $(m,n)$-obstruction for $G$, a contrary.
\end{proof}
\begin{cor}
Let $G$ be a non-simple  $\mathcal{T}(w(G),2)$-group. Then for proper non-trivial normal subgroup $N$ of $G$, $\frac{G}{N}$ is an $\mathcal{T}(w(G),1)$-group and so $w(\frac{G}{N})<w(G)$.
\end{cor}

\begin{theorem}
Let $G$ be a solvable $\mathcal{T}(3,n)$-group {\rm (}or $\mathcal{T}(4,n)$-group and the order of $G$ is odd{\rm )}. Then the derived length of $G$, $d$ is at most $\log_2(2n)$.
\end{theorem}

\begin{proof}
We argue by induction on $d$. The case $d=1$ being obvious. Assume that $d\geq 2$ and so, by Lemma \ref{l32},  the group $\frac{G}{G^{(d-1)}}$, has solvability length $d-1$, is an $\mathcal{T}(3,n-t)$-group, where $2t\leq n$. Therefore $d-1\leq\log_2(2(n-t))<\log_2(2(n))$. Thus $d-1<\log_2(2n)$, so $d\leq log_2(2n)$, as wanted. (We note that the bound $\log_2(2n)$ is the best possible, as $S_3$ is an $\mathcal{T}(3,2)$-group and $d(S_3)=2=\log_2(4)$.)\\
Now if $G$ is $\mathcal{T}(4,n)$-group and the order of $G$ is odd, then by argument similar, the result follows (for proof it is enough to note that $\mathcal{T}(4,1)$-group of odd order is abelian).
\end{proof}

Note that the group $D_8$ is a solvable $\mathcal{T}(4,1)$-group with solvability length $2$, but $2\nleq \log_2(2)$. This example suggests that it may be necessary to restrict
ourselves to  groups with odd order in the above Theorem.

If $G$ is a finite group, then for each prime divisor $p$ of $|G|$, we denote by $v_p(G)$ the
number of Sylow $p$-subgroups of $G$.
\begin{lemma}\label{l31}
Let $G$ be a finite $\mathcal{T}(m,n)$-group and $p$ be a prime number dividing $|G|$ such that every two distinct Sylow $p$-subgroups of $G$ have trivial intersection, then $v_p(G)\leq mn-1$.
\end{lemma}
\begin{proof}
Since $G$ is an $\mathcal{T}(m, n)$-group, we have $w(G)<mn$. Now Lemma 3 of \cite{end} completes the proof.
\end{proof}
Now we
obtain a solvability criteria for $\mathcal{T}(m,n)$-groups in terms of $m$ and $n$.
\begin{thm}\label{t3}
Let $G$ be an $\mathcal{T}(m,n)$-group. Then we have\\
$(i)$ If  $mn\leq 21$, then $G$ {\rm (}not necessarily finite{\rm )} is solvable and this estimate is sharp.\\
$(ii)$ If  $mn\leq 60$ and $G$ is non-solvable finite group, then $G=A_5$ {\rm (}in fact, $A_5$ is the only non-solvable $\mathcal{T}(m,n)$-group, which $mn\leq 60${\rm )}.
\end{thm}

\begin{proof}
$(i)$ Since $G$ is an $\mathcal{T}(m,n)$-group, we have $w(G)<mn$. Then by Theorem 1.2 of \cite{zar5}, $G$ is solvable. This estimate is sharp, because $A_5$ is an $\mathcal{T}(22,1)$-group.\\
$(ii)$ We know that the alternating group $A_5$ has five Sylow $2$-subgroups of order $4$, ten Sylow $3$-subgroups of order $3$ and six Sylow $5$-subgroups of order $5$, that their intersections are trivial. From this we can follow that  $A_5$ is an $\mathcal{T}(22,1)$, $\mathcal{T}(22,2)$, $\mathcal{T}(17,3)$ and $\mathcal{T}(14,4)$-group.
For uniqueness, suppose, on the contrary, that there exists a non-Abelian
finite simple group not isomorphic to $A_5$ and of least possible order which is an $\mathcal{T}(m,n)$-group, which $mn\leq 60$.  Then by Proposition 3 of \cite{br}, Lemma \ref{l31} and by argument similar to the one in the proof of Theorem 1.3 of \cite{abd1} gives a contradiction in
each case.
\end{proof}

\begin{cor}
Every arbitrary $\mathcal{T}(21,i)$-group $G$ with $i\leq z$  is solvable, where $z=|Z(G)|$.
\end{cor}
In the follow we show that the influence the value of $n$ for the solvability of $\mathcal{T}(m,n)$-groups is more important than the value of $m$.
\begin{cor}
Let $G$ be an $\mathcal{T}(m,n)$-group. Each of the following conditions implies that $G$ is solvable.

 (a) $n=2$ and $m\leq 21$.
~~(b) $n=3$ and $m\leq 16$.
~~(c) $n=4$ and $m\leq 13$.
~~(d) $n=5$ and $m\leq 8$.
~~(e) $n=6$ and $m\leq 8$.
~~(f) $n=7$ and $m\leq 7$.
~~(g) $n=8$ and $m\leq 7$.

\end{cor}
\begin{proof}
 Suppose, on the contrary, that $G$ is a non-solvable group. By Theorem \ref{t3}, $G\cong A_5$. This is a contradiction, since $A_5$ is an $\mathcal{T}(22,2)$-group but is not an $\mathcal{T}(21,2)$-group.

By similar argument of Case $(a)$ the rest Cases would be proved, since $$A_5\in \mathcal{T}(17,3)\bigcap \mathcal{T}(14,4)\bigcap \mathcal{T}(9,5) \bigcap \mathcal{T}(9,6)\bigcap \mathcal{T}(8,7) \bigcap \mathcal{T}(8,8) \text{~~but}$$
 $$A_5\not\in \mathcal{T}(16,3)\bigcup \mathcal{T}(13,4) \bigcup \mathcal{T}(8,6)\bigcup \mathcal{T}(7,8).$$
\end{proof}

\end{document}